\documentclass[a4paper]{amsart}
\usepackage{amsmath,amstext,amssymb,amsfonts,amscd,amsthm}
\usepackage[numbers,sort&compress]{natbib}
\usepackage{mathdots,mathrsfs,enumerate}
\usepackage{extarrows}
\usepackage{mathrsfs}  
\usepackage{dsfont}  
\usepackage{graphicx,color}
\usepackage{tikz}
\usepackage{stfloats}
\usetikzlibrary{3d,calc,patterns,graphs,arrows}
\usepackage{tikz-cd,array,diagbox}
\usepackage{changepage}
\usepackage{geometry}
\usepackage[CJKbookmarks=true,unicode,colorlinks,linkcolor=blue,anchorcolor=blue,citecolor=blue]{hyperref}
\usepackage{subfigure}
\numberwithin{equation}{section}

\usepackage{xcolor}
\definecolor{BeanPasteGreen}{RGB}{200,232,200}
\definecolor{WaterBlue}{RGB}{185,220,237}

\newtheorem{lemma}{Lemma}[section]
\newtheorem{theorem}{Theorem}[section]

\newtheorem{remark}{Remark}[section]

\newenvironment{proof*}
\numberwithin{equation}{section}

\newcommand{\op}[1]{\operatorname{#1}} 

\newcommand{\p}{\partial}

\renewcommand{\t}{\triangle}

\newcommand{\ls}{\lesssim}

\newcommand{\sm}{\setminus}

\newcommand{\abs}[1]{\left\vert#1\right\vert}
\newcommand{\norm}[1]{\left\Vert#1\right\Vert}
\newcommand{\set}[1]{\left\{#1\right\}}

\newcommand{\frb}[1]{\left(#1\right)}

\newcommand{\ol}{\overline}

\newcommand{\wt}{\widetilde}

\newcommand{\Gra}{\Longrightarrow}

\newcommand{\Z}{\mathbb Z}

\newcommand{\R}{\mathbb R}

\renewcommand{\a}{\alpha}
\renewcommand{\b}{\beta}
\newcommand{\g}{\gamma}
\providecommand{\G}{}
\renewcommand{\G}{\Gamma}
\renewcommand{\d}{\delta}

\newcommand{\e}{\epsilon}
\newcommand{\ve}{\varepsilon}

\renewcommand{\th}{\theta}

\renewcommand{\k}{\kappa}

\renewcommand{\l}{\lambda}

\newcommand{\s}{\sigma}

\newcommand{\vp}{\varphi}

\newcommand{\om}{\omega}
\newcommand{\Om}{\Omega}

\allowdisplaybreaks[3]

\begin{document}

\title[Interior Hessian estimates for the quadratic Hessian equation]
{Interior Hessian estimates for the quadratic Hessian equation}

\author{Zhisu Li}
\address{School of Mathematics and Center for Nonlinear Studies,
Northwest University, Xi'an, 710127, People's Republic of China}
\email{lizhisu@nwu.edu.cn}

\author{Ke Wu}
\address{School of Mathematics and Center for Nonlinear Studies, 
Northwest University, Xi'an, 710127, People's Republic of China}
\email{wuke@med.nwu.edu.cn}

\date{September 6, 2026.}

\maketitle
\tableofcontents

\begin{abstract}
We establish interior Hessian estimates
for solutions of the quadratic Hessian equation in arbitrary dimensions.
\end{abstract}




\section{Introduction}

In this paper, we study the interior Hessian estimate for the quadratic Hessian equation
\begin{equation}\label{eqn.sigma-2}
\s_2(D^2u)=\sum_{1\leq i<j\leq n}\l_i\l_j=1
\end{equation}
in dimension $n\geq2$,
where $\l_1,\l_2,\dots,\l_n$ are the eigenvalues of the Hessian matrix $D^2u$ of the function $u$.
Building upon
the gradient estimate by Trudinger \cite{Tru97} and also by Chou--Wang \cite{CW01},
and the Pogorelov estimate by Chou--Wang \cite{CW01},
we prove the following result.

\medskip

\begin{theorem}\label{thm.smooth}
Let $n\geq2$.
Suppose that $u\in C^\infty(B_2)$ satisfies \eqref{eqn.sigma-2} in $B_2\subset\R^n$.
Then
\[
\norm{D^2u}_{L^\infty(B_{1/2})}
\leq C\frb{n,\norm{u}_{L^\infty(B_2)}}.
\]
\end{theorem}

\begin{remark}
The interior regularity of viscosity solutions to \eqref{eqn.sigma-2}
on the positive branch $\t u>0$ in arbitrary dimensions
follows directly by combining Theorem \ref{thm.smooth}
with smooth Dirichlet approximation and standard compactness arguments.
See \cite[page 2475]{M21} and \cite[Remark 5.1]{SY25} for related details.
\end{remark}

\begin{remark}[Liouville theorem]
Theorem \ref{thm.smooth} implies that
every entire solution to \eqref{eqn.sigma-2}
with quadratic growth is a quadratic polynomial.
\end{remark}

\begin{remark}
For $n\geq2$, the corresponding result for the Hessian quotient equation
$\frac{\s_2(D^2u)}{\s_1(D^2u)}=1$ in $B_2\subset\R^n$ with $\t u>0$
follows directly from Theorem \ref{thm.smooth} as well.
Indeed, setting $v:=u-\frac{1}{2(n-1)}|x|^2$,
we have $\s_2(D^2v)=\frac{n}{2(n-1)}$;
see \cite{LS26}.
\end{remark}

\begin{remark}
The argument in the proof of Theorem \ref{thm.smooth} can also be
adapted to smooth solutions of $\s_2(D^2u)=f(x,u,Du)>0$ in $B_2$
with $f\in C^{1,1}(B_2\times\R\times\R^n)$.
Such solutions admit an interior Hessian estimate
depending only on $n$, $\norm{u}_{C^1(B_2)}$, $\norm{f}_{C^{1,1}}$, and
$\norm{1/f}_{L^\infty}$.
Moreover, if $f$ is independent of $Du$,
it suffices to assume that $f\in C^{0,1}(B_2\times\R)$,
and the interior Hessian estimate depends only on
$n$, $\norm{u}_{L^\infty(B_2)}$, $\norm{f}_{C^{0,1}}$,
and $\norm{1/f}_{L^\infty}$.
See Remark \ref{rmk.sigma-2=f} for details.
After the first version of this paper appeared,
Chen--Zhou--Zhu \cite{CZZ26}
released a work which applies the separation-propagation method developed here
to the case of positive $C^\alpha$ right-hand sides.
\end{remark}

\begin{remark}
Very recently,
Qiu--Yan \cite{QY26}
released a work which applies 
the separation-propagation method developed here
to the graphical scalar curvature equation $\s_2(\k)=1$.
We note that they have given a careful
and thorough discussion of this method.
\end{remark}

\medskip

The interior Hessian estimate for the quadratic Hessian equation \eqref{eqn.sigma-2} has a long history.

In dimension two, \eqref{eqn.sigma-2} is the Monge--Amp\`ere equation.
The classical interior estimate goes back to Heinz \cite{H59} via isothermal coordinates;
alternative proofs were later given by Chen--Han--Ou \cite{CHO16} using the maximum principle
and by Liu \cite{Liu21} using the partial Legendre transform.

In dimension three,
Warren--Yuan \cite{WY09}
exploited the minimal surface structure and a full strength Jacobi inequality
to obtain the interior Hessian estimate.
Qiu \cite{Q24} observed that the gradient graph $(x,Du)$ has bounded mean curvature
in $\frb{\R^3\times\R^3,fdx^2+dy^2}$
and combined this structure
with a Jacobi inequality to obtain an interior Hessian estimate
for positive right-hand sides $f\in C^{1,1}$.

In dimension four,
a major advance was made by Shankar--Yuan \cite{SY25}.
They resolved the unrestricted interior Hessian estimate problem,
gave a new proof of the three-dimensional result,
and obtained higher-dimensional estimates
under the dynamic semiconvexity condition
\begin{equation}\label{eqn.lambda-max-geq}
\l_{\min}(D^2u)\geq-c(n)\t u.
\end{equation}
Their proof synthesized Qiu's doubling idea \cite{Q24},
the Alexandrov-type differentiability approach of \cite{CT05},
and Savin's small perturbation theorem \cite{S07}.
Fan \cite{Fan26} extended the Shankar--Yuan estimates
to positive $C^{1,1}$ variable right-hand sides.

In arbitrary dimensions,
Guan--Qiu \cite{GQ19} proved a pointwise Hessian estimate
under the structural condition $\s_3(D^2u)\geq-A$,
which in particular covers convex solutions.
McGonagle--Song--Yuan \cite{MSY19} obtained the Hessian estimate
under an almost convexity condition by a compactness argument,
while Shankar--Yuan \cite{SY20} extended the result to semiconvex smooth solutions
by an integral method.
Using an improved regularity property of the equation
satisfied by the Legendre--Lewy transform,
Shankar--Yuan \cite{SY21} subsequently obtained
interior regularity for almost convex viscosity solutions.
Recently,
Chen--Jian--Tu--Zhou \cite{CJTZ26} and Zhou--Zhu \cite{ZZ26}
established interior regularity for convex viscosity solutions of $\s_2(D^2u)=f(x)>0$,
obtaining $C^2$ regularity for $f\in C^{0,1}$
and $C^{2,\a}$ regularity for $f\in C^\a$, respectively.

A different line of argument,
based on the foundational Dirichlet solvability theory
of Caffarelli--Nirenberg--Spruck \cite{CNS85}
and the pioneering Pogorelov estimate of Chou--Wang \cite{CW01},
was developed by Mooney.
In \cite{M21},
he proved strict $2$-convexity for convex viscosity subsolutions of $\s_2(D^2u)\geq1$,
and then obtained interior regularity for convex viscosity solutions.
Recently,
this strict $2$-convexity approach was also used in \cite{CJTZ26}
to construct a $2$-convex barrier for their approximation scheme.
In \cite{M25},
Mooney obtained an interior $C^2$ estimate
in terms of the $W^{2,p}$ norm for any $p>2$,
and ruled out several possible types of Hessian blow-up.
He also emphasized in \cite[page 2]{M25} the potential of the Chou--Wang estimate
for the regularity theory of the $2$-Hessian equation
by understanding ``the correct notion of strict $2$-convexity''.
In addition, 
\cite[Remark 2.6]{M25} observed 
from the proof of the Chou--Wang estimate in \cite{CW01}
that it suffices to assume $L_uw\geq0$, 
instead of requiring $w$ to be $2$-convex.
In the first arXiv version of this paper,
we used the corresponding formulation of the Chou--Wang estimate
given in \cite[Theorem 2.5]{M25}.

\medskip

Liouville-type rigidity for the $\s_2$ equation has also been extensively studied.
Warren--Yuan \cite{WY09} proved that entire solutions with quadratic growth
in dimension three are quadratic polynomials.
This was extended by Chang--Yuan \cite{CY10} under an almost convexity assumption
and Shankar--Yuan \cite{SY22} later extended it to general semiconvex entire solutions.
Shankar--Yuan \cite{SY25}
further obtained rigidity for both branches in dimension four,
and in higher dimensions
under condition \eqref{eqn.lambda-max-geq} on the positive branch
or the symmetric one $\l_{\max}(D^2u)\leq-c(n)\t u$ on the negative branch.
In contrast, Warren \cite{W16} constructed nonpolynomial entire saddle solutions for the $\s_2$ equation.
Other related rigidity results can be found in \cite{BCGJ03,CX19,Y02,LS26,MY26,LW26,JL26}
and the references therein.

\medskip

The Jacobi inequality plays a central role in essentially all previous works
on Hessian estimates without convexity assumptions.
In dimensions $n\geq 5$, however, this approach reaches a fundamental obstacle:
as observed in \cite{SY25}, even the subharmonicity of $\log\t u$
---the very basis of the Jacobi inequality---
fails in the absence of an a priori conditions on the Hessian.
This technical barrier has fostered a prevailing view that
``singular viscosity solutions'' may exist for $\sigma_2(D^2u)=1$,
and that the optimal regularity in higher dimensions may be only partial.
See \cite[page 492]{SY25} and \cite[page 2]{M25}.

Contrary to this view,
Theorems \ref{thm.smooth} establish full interior $C^2$ regularity
without any convexity or dimension restriction.
The unrestricted positive branch in dimensions $n\geq 5$,
which remained open despite the four-dimensional breakthrough of Shankar--Yuan \cite{SY25},
is therefore settled affirmatively.

\medskip

Our proof does not rely on the Jacobi inequality or the Hessian doubling argument.
Instead, we seek a suitable comparison function $w$
and a comparison domain $\Om$ with $B_{1/2}\subset\Om\Subset B_{3/2}$
for which the Pogorelov estimate of Chou--Wang \cite{CW01}
\[
\sup_\Om (w-u)^4|D^2u|
\leq C\frb{n,\norm{Du}_{L^\infty(\Om)},\norm{Dw}_{L^\infty(\Om)}}
\]
can be invoked.
Our starting point is to compare $u$ with the $2$-convex solution $v$ of the Dirichlet problem
\[
\s_2(D^2v)=1/4\ \text{in}\ B_{3/2},
\quad
v=u\ \text{on}\ \p B_{3/2}.
\]
The existence of such a smooth function $v$ is guaranteed by \cite{CNS85}.
This choice makes $v$ a super-solution of the equation for $u$,
so the comparison principle gives $v\geq u$ in $\ol{B_{3/2}}$.
Moreover, the concavity of $\s_2^{1/2}$ implies
\[
L_v(v-u)
\leq\s_2^{1/2}(D^2v)-\s_2^{1/2}(D^2u)
=-1/2<0,
\]
where $L_v:=(\s_2)_{ij}(D^2v)\p_{ij}$ denotes the linearized operator of $\s_2$ at $v$.
The local maximum principle therefore gives
\[
v-u>0\quad\text{in}\ B_{3/2}.
\]
This suggests seeking a uniform lower bound $v-u\geq2\d$ in $B_{1/2}$,
then shifting $v$ downward to $w:=v-\d$
and taking the component of $\{w-u>0\}$ containing $B_{1/2}$
as the comparison domain $\Om$.
To achieve this,
we develop a quantitative separation-propagation method.
For the initial separation,
the key observation is that $v-u$ satisfies the exact linear equation
\[
L_{(v+u)/2}(v-u)
=\s_2(D^2v)-\s_2(D^2u)
=-\frac34,
\]
which follows from the quadratic structure of $\s_2$.
The divergence-free structure of the coefficients of $L_{(v+u)/2}$
allows us to integrate by parts
and obtain the reverse Chou--Wang estimate
\[
1\ls\int_{B_1}(v-u)\t(v+u).
\]
With the interior gradient estimate,
the localized form of this estimate yields the seed separation
\[
v-u\geq\ve_0>0
\quad\text{on a small ball contained in}\ B_{1/2}.
\]
To propagate this separation, we construct a barrier for $L_v$.
For each $y\in B_{1/2}$, introduce
\[
\vp_y(x)
:=2\frb{(x-y)\cdot Dv(x)-v(x)+v(y)}
+\frac{\a}{2}|x-y|^2
-2\b|Dv(x)|^2
\]
and the associated exponential auxiliary function
\[
\psi_y:=e^{(c_y-\vp_y)/\g}-1.
\]
For some small $r>0$,
the parameters are chosen so that 
\[
L_v\psi_y>0
\quad\text{in an annular domain}.
\]
In establishing this positivity,
the term $-2\b|Dv|^2$ contributes a positive quadratic term in the eigenvalues,
which handles the case where a negative eigenvalue makes
the two terms in $(\vp_y)_i$ nearly offset.
Recalling $L_v(v-u)<0$, and using $\psi_y$ as a barrier function,
we obtain the main separation-propagation estimate
\begin{equation}\label{eqn.propagation-of-separation}
v-u\geq\e\ \text{in}\ B_r(y)
\ \Gra\
v-u\geq\th\e\ \text{in}\ B_{2r}(y)
\end{equation}
where $\th\in(0,1)$ is a uniform constant.
Iterating along a finite chain of intersecting balls gives
\[
v-u\geq2\d
\quad\text{in}\ B_{1/2}
\]
for a uniform $\d>0$.
Now, to complete the construction of the comparison domain,
we establish a boundary decay estimate,
which shows that $\{w>u\}$ stays a uniform distance away from $\p B_{3/2}$.
Hence, its component containing $B_{1/2}$ gives the desired domain $\Om$,
and thus
\[
w>u\ \text{in}\ \Om,
\qquad
w=u\ \text{on}\ \p\Om,
\qquad
w-u\geq\d\ \text{in}\ B_{1/2}.
\]
The Chou--Wang estimate then yields the desired Hessian bound.
Note that only $L_v$ and $L_{(u+v)/2}$ are used explicitly in the construction,
while $L_u$ appears only through the Chou--Wang estimate.

\medskip

The construction of $\vp_y$ is inspired by and modified from the radial-derivative test functions
of Shankar \cite[page 346]{S26}, Guan--Qiu \cite[page 8]{GQ19} and Qiu \cite[page 584]{Q24},
while that of $\psi_y$ draws on the Korevaar-type exponential cutoffs
employed in Korevaar \cite[page 415]{K87} and Shankar \cite[page 346]{S26}.
In proofs of interior Hessian estimates,
such auxiliary functions are often combined with Jacobi inequalities
to establish a doubling inequality.
For example, in dimension four,
Shankar--Yuan \cite[Propositions 2.1 and 3.1]{SY25}
establish an ``almost'' Jacobi inequality for $\ln\t u$
and then combine it with a suitable test function to obtain
\[
\sup_{B_2}\t u
\ls\sup_{B_1}\t u.
\]
Our separation-propagation argument can be carried out in a similar way.
To see this, a key observation is that
$L_v(v-u)\leq-1/2$ directly gives the following concavity Jacobi inequality
\[
L_va
\geq\frac{2}{a}A_{ij}(D^2v)a_i a_j+\frac12a^2
\quad
\text{for}\ a:=(v-u)^{-1}.
\]
Combining this with the barrier $\psi_y$ and the maximum principle implies the ``doubling inequality''
\[
\sup_{B_{2r}(y)}a\ls\sup_{B_r(y)}a.
\]
which is equivalent to the propagation of the separation \eqref{eqn.propagation-of-separation}.

\section{Preliminaries}

In this section,
we fix notation and recall several standard facts
about the quadratic Hessian equation
that will be used throughout the paper.

For $\l=(\l_1,\l_2,\dots,\l_n)\in\R^n$, set
\[
\s_1(\l):=\sum_{i=1}^n\l_i,
\quad
\s_2(\l):=\sum_{1\leq i<j\leq n}\l_i\l_j,
\quad
\G_2:=\set{\l\in\R^n:\s_1(\l)>0,\ \s_2(\l)>0}.
\]
The operator $\s_2$ is elliptic on $\G_2$,
and $\s_2^{1/2}$ is concave there.
For any $n\times n$ real symmetric matrix $M$ with its eigenvalue-vector $\l(M)\in\G_2$,
we write $\s_2(M):=\s_2(\l(M))$ for short, and then define
\[
A_{ij}(M)
:=\frac{\p\s_2(M)}{\p M_{ij}}
=\op{tr}M\d_{ij}-M_{ij},
\quad
A(M):=(A_{ij}(M)).
\]
Then $A(M)$ is positive definite.
We say $u\in C^2(\Om)$ is $2$-convex if $\l(D^2u(x))\in\G_2$ for any $x\in\Om$.
For a smooth $2$-convex solution $u$ of
\eqref{eqn.sigma-2}, we call
\[
L_u:=A_{ij}(D^2u)\p_{ij}
\]
the linearized operator of $\s_2$ at $u$.
Note that $L_uu=A_{ij}(D^2u)u_{ij}=2\s_2(D^2u)$.

\medskip

Now, we introduce some basic algebraic inequalities of $\s_2$.

\begin{lemma}\label{lem.spectral}
Let $\l=(\l_1,\l_2,\dots,\l_n)\in\G_2$ satisfy $\s_2(\l)=1$.
Write $f(\l):=\s_2(\l)$,
and $f_i:=\p f/\p\l_i$ for $\l_1\geq\l_2\geq\cdots\geq\l_n$.
Then
\begin{gather*}
f_1(\l)\geq\frac{1}{\s_1(\l)},
\quad
f_i\geq\frb{1-\frac{1}{\sqrt2}}\s_1(\l)\
\text{for all}\ i\in\Z_{[2,n]},
\quad
f_1\l_1^2\geq\frac{\s_1(\l)}{n^2}.
\end{gather*}
\end{lemma}

\begin{remark}
The first two inequalities are contained in
\cite[Corollary 2.1]{SY25},
see also \cite[page 321, (16)]{LT94} for their sharper lower bounds
when $i\geq2$.
\end{remark}

\begin{proof}
Since $\l\in\G_2$ and $\s_2(\l)=1$, we have $\s_1(\l)^2=2+|\l|^2$
and $f_i=\s_1(\l)-\l_i>0$.
In particular,
\[
f_1
=\s_1(\l)-\l_1
=\frac{2+\sum_{i=2}^n\l_i^2}{\s_1(\l)+\l_1}
>\frac{1}{\s_1(\l)}.
\]
Moreover, since $\l_1\geq\s_1(\l)/n$,
one have $f_1\l_1^2\geq\s_1(\l)/n^2$.
For $i\geq2$, the desired estimate is immediate if $\l_i\leq0$.
If $\l_i>0$, then $2\l_i^2\leq i\l_i^2\leq\sum_{j=1}^i\l_j^2<\s_1(\l)^2$,
and hence
\[
f_i=\s_1(\l)-\l_i
>\frb{1-\frac1{\sqrt2}}\s_1(\l),
\quad\text{for all}\ i\in\Z_{[2,n]}.
\]
\end{proof}

Finally, we collect some classical results:
for the $\s_2$-equation:
the solvability of the Dirichlet problem \cite{CNS85},
the interior gradient estimate \cite{Tru97,CW01},
and the Pogorelov-type Hessian estimate \cite{CW01}.

\begin{lemma}[Classical solvability,
\texorpdfstring{\cite[Theorem 3]{CNS85}}{}]
\label{lem.solvability}
Let $r>0$, $c>0$ and $g\in C^\infty(\p B_r)$.
Then there exists a unique $2$-convex solution $u\in C^\infty(\ol{B_r})$ to the Dirichlet problem
\[
\s_2(D^2u)=c\ \textrm{in}\ B_r,
\quad
u=g\ \textrm{on}\ \p B_r.
\]
\end{lemma}

\begin{lemma}[Gradient estimate,
\texorpdfstring{\cite[Theorem 3.1]{Tru97}}{},
\texorpdfstring{\cite[Theorem 3.2]{CW01}}{}]
\label{lem.gradient-estimate}
Let $r>0$ and $c>0$.
Suppose $u\in C^\infty(B_r)$ is a $2$-convex solution of $\s_2(D^2u)=c$ in $B_r$.
Then
\[
\norm{Du}_{L^\infty(B_{r/2})}
\leq\frac{C(n)}{r}\norm{u}_{L^\infty(B_r)}.
\]
\end{lemma}

\begin{lemma}[Pogorelov estimate,
\texorpdfstring{\cite[Theorem 4.1]{CW01}}{}]
\label{lem.Pogorelov-estimate}
Let $\Om\subset B_2\subset\R^n$ be an open domain,
and $u\in C^\infty(\ol\Om)$
be a $2$-convex solution to $\s_2(D^2u)=1$ in $\Om$.
Suppose that there is a $2$-convex function $w\in C^\infty(\ol\Om)$ satisfies
$w>u$ in $\Om$ and $w=u$ on $\p\Om$.
Then
\[
\sup_\Om (w-u)^4|D^2u|
\leq C\frb{n,\norm{Du}_{L^\infty(\Om)},
\norm{Dw}_{L^\infty(\Om)}}.
\]
\end{lemma}

\section{Proof of the theorem}
\label{sec.proof-of-the-Thm}

Throughout this section,
let $u$ be a smooth solution of \eqref{eqn.sigma-2} in $B_2$
with $\t u>0$ in $B_2$.
We may assume that
\[
K_0:=1+\norm{u}_{L^\infty(B_2)}<+\infty.
\]
Applying Lemma \ref{lem.gradient-estimate} in $B_{1/2}(x)\subset B_2$
for $x\in B_{3/2}$ gives
\[
\norm{Du}_{L^\infty(B_{3/2})}\leq C(n)K_0.
\]
By Lemma \ref{lem.solvability},
there exists a unique smooth $2$-convex solution $v$ of
\begin{equation}\label{eqn.v}
\begin{cases}
\s_2(D^2v)=1/4 &\mathrm{in}\ B_{3/2},\\
v=u &\mathrm{on}\ \p B_{3/2}.
\end{cases}
\end{equation}
The comparison principle gives $u\leq v$ in $\ol{B_{3/2}}$.
Since $\l(D^2v)\in\G_2$, we have $\t v>0$ in $B_{3/2}$,
and hence the maximum principle gives
\[
\sup_{B_{3/2}}v\leq\sup_{\p B_{3/2}}v=\sup_{\p B_{3/2}}u\leq K_0.
\]
On the other hand, $v\geq u\geq-K_0$ in $B_{3/2}$.
Therefore $\norm{v}_{L^\infty(B_{3/2})}\leq K_0$.
Moreover,
for any $x\in B_1$,
applying Lemma \ref{lem.gradient-estimate}
in $B_{1/2}(x)\subset B_{3/2}$ gives
\begin{equation}\label{eqn.gradient-v-u}
2\norm{Dv}_{L^\infty(B_1)}+\norm{Du}_{L^\infty(B_1)}
\leq K,
\end{equation}
where $K=K(n,K_0)\geq1$.
In particular, $\norm{D(v-u)}_{L^\infty(B_1)}\leq K$.

\medskip

Now, we try to quantify the separation between $v$ and $u$.

\subsection{A seed separation}

\begin{lemma}\label{lem.seed}
There exist constants $r=r(n,K)\in(0,1/32)$, $\ve_0=\ve_0(n,K)>0$,
and a point $x_*\in\ol{B_{1/32}}$ such that
\[
v-u\geq\ve_0
\quad\text{in}\ \ol{B_r(x_*)}.
\]
\end{lemma}

\begin{proof}
Let $\om:=(u+v)/2$ in $B_{3/2}$.
Since $\G_2$ is convex,
we have $\l(D^2\om)\in\G_2$, and hence $A(D^2\om)>0$.
A direct calculation gives
\begin{align}
-\frac34
=\s_2(D^2v)-\s_2(D^2u)
&=\frac{1}{2}\frb{(\t v)^2-\op{tr}(D^2v)^2-\frb{(\t u)^2-\op{tr}(D^2u)^2}}
\nonumber\\
&=\frac12\frb{
\frb{\t v-\t u}\frb{\t v+\t u}
-\op{tr}\frb{(D^2v-D^2u)(D^2v+D^2u)}}
\nonumber\\
&=(\t\om)\frb{\t v-\t u}-\op{tr}\frb{D^2\om(D^2v-D^2u)}
\nonumber\\
&=\sum_{i,j=1}^n\frb{(\t \om)\d_{ij}-\om_{ij}}(v_{ij}-u_{ij})
\nonumber\\
&=\sum_{i,j=1}^nA_{ij}(D^2\om)(v_{ij}-u_{ij}).
\label{eqn.polarized}
\end{align}
Fix $\rho=1/32$.
Choose $\eta\in C_c^\infty(B_1)$ satisfying
$\eta\geq0$ in $B_1$, $\int_{B_1}\eta=1$ and $\norm{D^2\eta}_{L^\infty(B_1)}\leq C(n)$,
and set $\eta_\rho(x):=\rho^2\eta(x/\rho)$.
Then $\eta_\rho\in C_c^\infty(B_\rho)$, $\int_{B_\rho}\eta_\rho=\rho^{n+2}$
and $\norm{D^2\eta_\rho}_{L^\infty(B_\rho)}\leq C(n)$.
Multiplying \eqref{eqn.polarized} by $\eta_\rho$,
integrating by parts twice,
and using $\sum_{i=1}^n\p_iA_{ij}(D^2\om)=0$ in $B_\rho$ for all $j\in\Z_{[1,n]}$,
we obtain
\begin{align*}
\frac34\rho^{n+2}
&=\frac34\int_{B_\rho}\eta_\rho
=-\int_{B_\rho}\sum_{i,j=1}^n\eta_\rho A_{ij}(D^2\om)(v_{ij}-u_{ij})
\\
&=-\int_{B_\rho}(v-u)\sum_{i,j=1}^nA_{ij}(D^2\om)\p_{ij}\eta_\rho
\\
&\leq C\sup_{B_\rho}(v-u)\int_{B_\rho}\op{tr}A(D^2\om).
\end{align*}
Moreover,
the divergence theorem and the gradient estimate \eqref{eqn.gradient-v-u} give
\[
\int_{B_\rho}\op{tr}A(D^2\om)
=\int_{B_\rho}(n-1)\t \om
=(n-1)\int_{\p B_\rho}
D\om\cdot\frac{x}{|x|}
\leq C(n,K)\rho^{n-1}.
\]
Thus
\[
\sup_{B_\rho}(v-u)
\geq c_0
\]
for some $c_0=c_0(n,K)>0$.
Choose $x_*\in\ol{B_\rho}$ such that $v(x_*)-u(x_*)\geq c_0$.
By the gradient estimate \eqref{eqn.gradient-v-u},
we may take $r=r(n,K)\in(0,1/32)$ sufficiently small
and $\ve_0=\ve_0(n,K)>0$ such that,
for any $x\in\ol{B_r(x_*)}$,
\[
v(x)-u(x)
\geq v(x_*)-u(x_*)-\norm{D(v-u)}_{L^\infty(B_1)}|x-x_*|\\
\geq c_0-Kr
>\ve_0.
\]
Finally, since $|x_*|\leq1/32$ and $r<1/32$,
we have $|x_*|+r<1/16$.
Hence $\ol{B_r(x_*)}\subset B_{1/16}$,
and the proof is complete.
\end{proof}

\medskip

The preceding lemma gives a quantitative separation only on one small ball.
To propagate this information through the interior,
we next construct a barrier of the linearized operator $L_v$.
\subsection{A barrier function of the linearized operator $L_v$}\label{subsec.construct} 

Keep $r=r(n,K)\in(0,1/32)$ from Lemma \ref{lem.seed}.
Choose $\b=\b(n,K)>0$ sufficiently small such that
\[
\b K\leq\frac{r}{8\sqrt n},
\]
and then choose $\a=\a(n,K)\geq2$ sufficiently large such that
\[
6\a r^2>4+12Kr+\frac{\b}{2}K^2
\quad\text{and}\quad
\b\frb{1-\frac{1}{\sqrt2}}\a\geq6n,
\]
and then choose $\g=\g(n,K)>0$ sufficiently small such that
\[
\frac{49r^2}{64n^3\g}>n\a
\quad\text{and}\quad
\frb{1-\frac1{\sqrt2}}\frac{\a^2r^2}{4n\g}>n\a.
\]

\medskip

Fix any $y\in B_{1/2}$.
Then $B_{4r}(y)\Subset B_1$.
Consider the auxiliary function
\begin{equation}\label{eqn.varphi}
\vp_y(x)
:=2\frb{(x-y)\cdot Dv(x)-v(x)+v(y)}
+\frac{\a}{2}|x-y|^2-2\b\abs{Dv(x)}^2.
\end{equation}
The gradient estimate \eqref{eqn.gradient-v-u} gives
\begin{equation}\label{eqn.g-bound}
\abs{2\frb{(x-y)\cdot Dv(x)-v(x)+v(y)}}\leq 2K|x-y|
\quad\text{for any}\
x\in B_1.
\end{equation}
Then
\begin{gather}
\vp_y\geq-8Kr-\frac{\b}{2}K^2\quad\text{in}\ B_{4r}(y),\
\label{eqn.vpy-lower}
\\
a_y:=\sup_{\ol{B_{2r}(y)}}\vp_y\leq 4Kr+\frac{\a}{2}(2r)^2.
\label{eqn.ay-upper}
\end{gather}
With the choice of $\a$ above,
\begin{gather*}
\inf_{\p B_{4r}(y)}\vp_y
\geq-8Kr+\frac{\a}{2}(4r)^2-\frac{\b}{2}K^2
>4Kr+\frac{\a}{2}(2r)^2+4
\geq a_y+4=:c_y.
\end{gather*}
Then
\[
\vp_y<c_y
\quad\text{in}\ \ol{B_{2r}(y)},
\quad
\vp_y>c_y
\quad\text{on}\ \p B_{4r}(y).
\]
Let $\Om_y$ be the connected component of $\set{\vp_y<c_y}$ containing $B_{2r}(y)$.
Then
\begin{equation}\label{eqn.Omega-y}
B_{2r}(y)\subset\Om_y\Subset B_{4r}(y).
\end{equation}
Next, set the associated exponential auxiliary function
\begin{equation}\label{eqn.w-y}
\psi_y:=e^{(c_y-\vp_y)/\g}-1
\quad\text{in}\ \Om_y.
\end{equation}
Then $\psi_y>0$ in $\Om_y$ and $\psi_y=0$ on $\p\Om_y$.

\medskip

The key point at this moment is to show that
$\psi_y$ is a sub-solution of $L_v$ outside $B_r(y)$.

\begin{lemma}\label{lem.cap-inequality}
For any $y\in B_{1/2}$, we have
\begin{equation}\label{eqn.Lv-w-positive}
L_v\psi_y>0\quad\text{in}\ \Om_y\sm\ol{B_r(y)}.
\end{equation}
\end{lemma}

\begin{proof}
Fix any $x\in\Om_y\sm\ol{B_r(y)}$.
After a rotation of coordinates, we may assume that
\[
D^2v(x)
=\frac12\,\op{diag}(\l_1,\l_2,\ldots,\l_n),
\quad
\l_1\geq\l_2\geq\cdots\geq\l_n.
\]
All quantities below are evaluated at $x$.
Since $L_vv=2\s_2(D^2v)=\frac12$
and $\sum_{i,j=1}^nA_{ij}(D^2v)v_{ijk}=0$ for all $k\in\Z_{[1,n]}$,
a direct calculation gives
\begin{align*}
(\vp_y)_i
&=\frb{x_i-y_i-2\b v_i}\l_i+\a(x_i-y_i),
\\
(\vp_y)_{ij}
&=2\frb{v_{ij}+\sum_{k=1}^n(x_k-y_k)v_{ijk}}+\a\d_{ij}
-4\b\sum_{k=1}^n v_{kj}v_{ki}
-4\b\sum_{k=1}^nv_kv_{ijk},
\\
L_v\vp_y
&=\frac12\frb{2+(n-1)\a\s_1(\l)-\b\sum_{i=1}^nf_i\l_i^2},
\\
L_v\psi_y
&=\frac{\psi_y+1}{2\g}
\frb{\frac{1}{\g}\sum_{i=1}^nf_i(\vp_y)_i^2
+\b\sum_{i=1}^nf_i\l_i^2-2-(n-1)\a\s_1(\l)},
\end{align*}
where $f_i:=\p\s_2/\p_{\l_i}$ as in Lemma \ref{lem.spectral}.
Thus, to show $L_v\psi_y>0$, it suffices to prove
\begin{equation}\label{eqn.key-cap-target}
\frac{1}{\g}
\sum_{i=1}^nf_i(\vp_y)_i^2
+\b\sum_{i=1}^nf_i\l_i^2
>2+(n-1)\a\s_1(\l).
\end{equation}
By $\l\in\G_2$ and $2=2\s_2(\l)=\s_1(\l)^2-\sum_{i=1}^n\l_i^2$,
we have $\s_1(\l)\geq\sqrt2$, and hence
\begin{equation}\label{eqn.rhs-bound}
2+(n-1)\a\s_1(\l)<n\a\s_1(\l).
\end{equation}

Since $|x-y|\geq r$, there exists $j\in\Z_{[1,n]}$ such that
\[
|x_j-y_j|\geq|x-y|/\sqrt n\geq r/\sqrt n.
\]
Set $z_j:=x_j-y_j-2\b v_j$.
Then $(\vp_y)_j=z_j\l_j+\a(x_j-y_j)$.
The gradient estimate \eqref{eqn.gradient-v-u} and the choice of $\b$ give
\[
2\b|v_j|
\leq\b K
\leq\frac{r}{8\sqrt n}
\leq\frac18|x_j-y_j|,
\]
and therefore,
\begin{gather}
\frb{x_j-y_j}z_j
=(x_j-y_j)^2-2\b\frb{x_j-y_j}v_j
\geq\frac{7}{8}|x_j-y_j|^2>0,
\label{eqn.zj-bound}\\
\frac{7}{8}|x_j-y_j|
\leq|z_j|
\leq\frac{9}{8}|x_j-y_j|.
\nonumber
\end{gather}
Next, we distinguish three cases.

\medskip

\textit{Case 1: $j=1$.}
Since $\l_1>0$ and $z_1(x_1-y_1)>0$,
the two terms in $(\vp_y)_1$ have the same sign.
Thus \eqref{eqn.zj-bound} gives
\[
|(\vp_y)_1|\geq\frac{7}{8}|x_1-y_1|\l_1\geq\frac{7r}{8\sqrt n}\l_1.
\]
Hence, by Lemma \ref{lem.spectral} and the choice of $\g$,
\begin{align*}
\frac{1}{\g}\sum_{i=1}^nf_i(\vp_y)_i^2
\geq\frac{49r^2}{64n\g}f_1\l_1^2
\geq\frac{49r^2}{64n^3\g}\s_1(\l)
>n\a\s_1(\l).
\end{align*}

\medskip

\textit{Case 2: $j\geq2$ and $|(\vp_y)_j|\geq\frac{\a}{2}|x_j-y_j|$.}
By Lemma \ref{lem.spectral},
it follows that $f_j\geq\frb{1-\frac1{\sqrt2}}\s_1(\l)$.
The choice of $\g$ gives
\begin{align*}
\frac{1}{\g}
\sum_{i=1}^nf_i(\vp_y)_i^2
\geq\frb{1-\frac1{\sqrt2}}\frac{\a^2r^2}{4n\g}\s_1(\l)
>n\a\s_1(\l).
\end{align*}

\medskip

\textit{Case 3: $j\geq2$ and $|(\vp_y)_j|<\frac{\a}{2}|x_j-y_j|$.}
Then
\[
|z_j\l_j|\geq\a|x_j-y_j|-|(\vp_y)_j|
>\frac{\a}{2}|x_j-y_j|
\geq\frac{4\a}{9}|z_j|,
\]
so that $|\l_j|>\frac{4}{9}\a$.
Using Lemma \ref{lem.spectral} and the choice of $\a$, we obtain
\begin{align*}
\b\sum_{i=1}^nf_i\l_i^2
\geq\b f_j\l_j^2
>\b\frb{1-\frac{1}{\sqrt2}}\s_1(\l)\cdot\frac{16}{81}\a^2
> n\a\s_1(\l).
\end{align*}

In all three cases, \eqref{eqn.key-cap-target} follows from \eqref{eqn.rhs-bound}.
Hence $L_v\psi_y>0$ in $\Om_y\sm\ol{B_r(y)}$.
\end{proof}

\medskip

In view of Lemma \ref{lem.cap-inequality},
the auxiliary function $\psi_y$ can now be used as a barrier for $v-u$.
This gives a quantitative propagation of the lower bound of $v-u$
from $B_r(y)$ to the larger ball $B_{2r}(y)$.

\subsection{Propagation of the separation}\label{subsec.propagation}

\begin{lemma}\label{lem.one-step}
Let $y\in B_{1/2}$ and $\e>0$.
If $v-u\geq\e$ in $\ol{B_r(y)}$,
then there exists $\th=\th(n,K)\in(0,1)$ such that
\[
v-u\geq\th\e\quad\text{in}\ \ol{B_{2r}(y)}.
\]
\end{lemma}

\begin{proof}
Let $U_y$ be the connected component of $\Om_y\sm\ol{B_r(y)}$
containing $B_{2r}(y)\sm\ol{B_r(y)}$.
Then
\[
\p U_y\subset\p\Om_y\cup\p B_r(y).
\]
Using $\Om_y\subset B_{4r}(y)$,
the lower bound \eqref{eqn.vpy-lower} of $\vp_y$ in $B_{4r}(y)$
and the upper bound \eqref{eqn.ay-upper} of $a_y$,
we obtain
\[
0<c_y-\vp_y=a_y+2-\vp_y\leq C(n,K)
\quad\text{in}\ \Om_y.
\]
On the other hand, $c_y-\vp_y\geq c_y-a_y=4$ in $\ol{B_{2r}(y)}$.
Therefore, recalling the definition
\[
\psi_y=e^{(c_y-\vp_y)/\g}-1,
\]
there exist constants $c_0=c_0(n,K)>0$ and $C_0=C_0(n,K)>0$ such that
\[
0<\psi_y\leq C_0
\quad\text{in}\ \Om_y,
\qquad
\psi_y\geq c_0
\quad\text{in}\ \ol{B_{2r}(y)}.
\]
Thus, $v-u\geq\e\geq\frac{\e}{C_0}\psi_y$ on $\p U_y\cap\p B_r(y)$.
Moreover,
$v-u\geq0=\psi_y$ on $\p U_y\cap\p\Om_y$.
Consequently,
\[
(v-u)\geq \frac{\e}{C_0}\psi_y\quad\text{on}\ \p U_y.
\]
By the concavity of $\s_2^{1/2}$ on $\G_2$
and $\s_2(D^2v)=1/4$,
we have
\begin{equation}\label{eqn.L-v-u<-1/2}
1=\s_2(D^2u)^{1/2}
\leq\s_2(D^2v)^{1/2}+\frac{1}{2\s_2(D^2v)^{1/2}}A_{ij}(D^2v)(u-v)_{ij}
=\frac12-L_v(v-u)
\quad\text{in}\ B_{3/2}.
\end{equation}
Together with Lemma \ref{lem.cap-inequality}, this gives
\[
L_v(v-u)\leq-\frac{1}{2}<0<\frac{\e}{C_0}L_v\psi_y
\quad\text{in}\ U_y.
\]
The maximum principle therefore yields $v-u\geq\frac{\e}{C_0}\psi_y$ in $\ol{U_y}$.
Since $\psi_y\geq c_0$ in $\ol{B_{2r}(y)}$, we obtain
\[
v-u\geq\frac{\e}{C_0}\psi_y\geq\frac{c_0}{C_0}\e
\quad\text{in}\ B_{2r}(y)\sm\ol{B_r(y)}.
\]
Set $\th:=\min\set{\frac{1}{2},\frac{c_0}{C_0}}$.
Then $v-u\geq\e\geq\th\e$ in $\ol{B_r(y)}$.
Thus $v-u\geq\th\e$ in $\ol{B_{2r}(y)}$.
\end{proof}

\medskip

Iterating Lemma \ref{lem.one-step} along a finite chain of intersecting balls,
we obtain a uniform positive lower bound for $v-u$ in $B_{1/2}$.

\begin{lemma}[Uniform interior separation]\label{lem.separation}
There exists $\ve_1=\ve_1(n,K)>0$ such that
\begin{equation}\label{eqn.uniform-separation}
v-u\geq\ve_1\quad\text{in}\ B_{1/2}.
\end{equation}
\end{lemma}

\begin{proof}
Recall from Lemma \ref{lem.seed} that there exist
$r=r(n,K)\in(0,1/32)$, $x_*\in\ol{B_{1/32}}$ and $\ve_0=\ve_0(n,K)>0$ such that
\[
v-u\geq\ve_0\quad\text{in}\ \ol{B_r(x_*)}.
\]
Given any $y\in B_{1/2}$,
choose points $y_0:=x_*$, $y_1$, $y_2$, $\dots$, $y_N:=y$
on the line segment from $x_*$ to $y$ so that
\[
|y_k-y_{k-1}|\leq r\ \text{for all}\ k\in\Z_{[1,N]},
\quad
N\leq N_*:=1+\frac{1}{r}.
\]
Then $\set{y_k}_{k=0}^N\subset B_{1/2}$.
Repeated application of Lemma \ref{lem.one-step},
together with $B_r(y_{k+1})\subset B_{2r}(y_k)$
gives $v-u\geq\th^k\ve_0$ in $\ol{B_r(y_k)}$ for all $k\in\Z_{[0,N]}$.
In particular,
\[
(v-u)(y)\geq\ve_0\th^N\geq\ve_0\th^{N_*}.
\]
Thus \eqref{eqn.uniform-separation} holds with $\ve_1=\ve_0\th^{N_*}>0$.
\end{proof}

\medskip

To construct the desired comparison domain $\Om$,
we derive an upper bound of $v-u$ near $\p B_{3/2}$.
\subsection{Boundary decay}

\begin{lemma}\label{lem.boundary-decay}
There exists $C=C(n,K)$ such that
\begin{equation}\label{eqn.boundary-decay}
0\leq (v-u)(x)
\leq C d_x^{1/2}
\quad\text{for any}\ x\in B_{3/2},
\end{equation}
where $d_x$ denotes the distance from $x$ to $\p B_{3/2}$.
\end{lemma}

\begin{proof}
Let $h\in C^\infty(B_{3/2})\cap C^1(\ol{B_{3/2}})$ be the harmonic function
satisfying $h=v\,(=u)$ on $\p B_{3/2}$.
Then $\t v>0=\t h$ in $B_{3/2}$.
By the comparison principle, it follows that $v\leq h$ in $\ol{B_{3/2}}$.
Therefore
\[
0\leq v-u\leq h-u\quad\ \text{in}\ \ol{B_{3/2}}.
\]
By the global H\"{o}lder estimate for harmonic functions \cite[Proposition 4.12]{CC95},
we have
\[
[h]_{C^{1/2}(\ol{B_{3/2}})}
\leq C(n)\|u\|_{C^1(\p B_{3/2})}
\leq C(n)K.
\]
For any $x\in B_{3/2}$, let $x_0\in\p B_{3/2}$ satisfy $|x-x_0|=d_x$.
Since $h(x_0)=u(x_0)$, we have
\[
h(x)-u(x)
\leq [h]_{C^{1/2}(\ol{B_{3/2}})}d_x^{1/2}+u(x_0)-u(x)
\leq C(n)K d_x^{1/2}+Kd_x
\quad\text{for any}\ x\in B_{3/2}.
\]
This completes the proof.
\end{proof}

\begin{remark}
In the first arXiv version of this article,
\eqref{eqn.boundary-decay} was proved
by using the Poisson integral formula for the harmonic function $h$.
Alternatively, one may prove it by taking $d_x^{1/2}=|x-x_0|^{1/2}$
as a barrier function (cf. \cite[pp. 33--34]{QY26}).
We also note that \eqref{eqn.boundary-decay} is a direct consequence
of the global gradient estimate \cite[Theorem 3.4]{CW01} for $\s_2(D^2v)=1/4$.
\end{remark}

\subsection{Proof of Theorem \ref{thm.smooth}}

By symmetry, we may assume that
$u$ lies on the positive branch $\t u>0$ in $B_2$.
Indeed, since $(\t u)^2=|D^2u|^2+2>0$ and $\t u$ is continuous,
either $\t u>0$ or $\t u<0$ in $B_2$.
Moreover, $\s_2(D^2(-u))=\s_2(D^2u)=1$, so the negative branch
is reduced to the positive one by replacing $u$ with $-u$.

Choose $d_0=d_0(n,K)\in(0,1/4)$ sufficiently small
such that $C(n,K)d_0^{1/2}\leq\ve_1/4$,
where $\ve_1=\ve_1(n,K)$ comes from Lemma \ref{lem.separation},
and $C(n,K)$ comes from Lemma \ref{lem.boundary-decay}.
Let $\d:=\d(n,K)=\ve_1/2$.
By Lemma \ref{lem.separation},
$v-u\geq\ve_1>\d$ in $B_{1/2}$,
Hence $B_{1/2}$ is contained in a connected component of $\set{v-u>\d}$,
which we denote by $\Om$.
For any $x\in B_{3/2}$,
if $d_x\leq d_0$, then $|x|=\frac32-d_x\geq\frac32-d_0>\frac54$.
Hence it follows from Lemma \ref{lem.boundary-decay} that
\begin{equation}
(v-u)(x)\leq\frac{\ve_1}{4},
\quad
\text{whenever}\ d_x\leq d_0.
\end{equation}
Thus, the set $\set{v-u>\d}$ cannot intersect $\set{d_x\leq d_0}$.
Therefore,
$B_{1/2}\subset\Om\Subset B_{3/2-d_0}$,
and hence that $B_{d_0/2}(x)\subset B_{3/2}$ for any $x\in\Om$.
Set $w:=v-\d$ in $B_{3/2}$.
Then $w$ is a $2$-convex function in $B_{3/2}$.
By the definition of $\Om$,
\[
w>u\quad\text{in}\ \Om,\
\quad w=u\quad\text{on}\ \p\Om.
\]
Applying the gradient estimate (\cite{Tru97,CW01}, Lemma \ref{lem.gradient-estimate})
to $\s_2(D^2v)=1/4$ in $B_{d_0/2}(x)$, we get
\[
|Dv(x)|
\leq C(n)d_0^{-1}\norm{v}_{L^\infty(B_{d_0/2}(x))}
\leq C(n,K).
\]
Thus
\[
\norm{Dw}_{L^\infty(\Om)}=\norm{Dv}_{L^\infty(\Om)}\leq C(n,K).
\]
Applying Pogorelov estimate (\cite{CW01}, Lemma \ref{lem.Pogorelov-estimate})
to $u$ and $w$ in $\Om$, we obtain
\[
(w-u)^4|D^2u|\leq C(n,K)\quad\text{in}\ \Om.
\]
Since, by Lemma \ref{lem.separation},
\[
w-u
=v-u-\d
\geq\ve_1-\d
=\d
\quad
\text{in}\ B_{1/2},
\]
we obtain
\[
\sup_{B_{1/2}}|D^2u|
\leq C(n,K)\sup_{B_{1/2}}(w-u)^{-4}
\leq C(n,K)\d^{-4}
\leq C(n,K_0).
\]
This completes the proof of Theorem \ref{thm.smooth}.
\hfill{$\qed$}

\begin{remark}\label{rmk.sigma-2=f}
The argument of Theorem \ref{thm.smooth} extends to the equation
\begin{equation}\label{eqn.sigma-2=f}
\s_2(D^2u)=f(x,u,Du)\quad\text{in}\ B_2,
\end{equation}
where $f\in C^{1,1}(B_2\times\R\times\R^n)$ and $f>0$.
After the normalization
\[
f_0:=\norm{1/f}_{L^\infty}^{-1},\quad
\wt u:=f_0^{-1/2}u,
\quad
\wt f(x,z,p):=
f_0^{-1}f(x,f_0^{1/2}z,f_0^{1/2}p),
\]
we may assume that
\[
f(x,u,Du)\geq1\quad\text{in}\ B_2\times\R\times\R^n.
\]
The interior gradient estimate of Chou--Wang
\cite[Theorem~3.2]{CW01} gives
\[
\norm{Du}_{L^\infty(B_{3/2})}
\leq
C\frb{n,\norm{u}_{L^\infty(B_2)},\norm{f}_{C^{0,1}(B_2\times\R\times\R^n)}}.
\]
With this estimate,
all arguments leading to the construction
of the comparison function $w=v-\d$
and of the comparison domain $\Om$
use only the conditions
\[
\s_2(D^2u)\geq1,
\quad
\t u>0
\quad
\text{in}\ B_2.
\]
Indeed, for the solution $v$ of the Dirichlet problem \eqref{eqn.v},
the comparison principle yields $u\leq v$ in $\ol{B_{3/2}}$.
Moreover,
the inequality $\s_2(D^2v)-\s_2(D^2u)\leq-\frac34$
leads to the seed separation lemma (Lemma \ref{lem.seed}),
while the concavity of $\s_2^{1/2}$ implies $L_v(v-u)\leq-1/2$
(see \eqref{eqn.L-v-u<-1/2}).
These are the only properties of $u$ needed
in the remaining argument of Section \ref{sec.proof-of-the-Thm},
so the same argument applies.
Furthermore, the corresponding Pogorelov estimate for equation \eqref{eqn.sigma-2=f}
follows from a minor modification of the argument of
Li--Ren--Wang \cite[Theorem 1]{LRW16},
by considering the auxiliary function
\[
\Phi(x,\xi)
:=(w-u)^\a\exp\set{\frac{\ve}{2}|Du|^2+\frac{a}{2}|x|^2}u_{\xi\xi}
\quad\text{for any}\ (x,\xi)\in\ol\Om\times\mathbb S^{n-1};
\]
and also follows from \cite[Proposition 2.6]{CJTZ26}
when $f\in C^{0,1}(B_2\times\R)$.
\end{remark}

\subsection*{Acknowledgments}

We are deeply grateful to Professor Yu Yuan for introducing us to this problem
and for teaching us many of the fundamental techniques of this subject.
We also thank Professors Yu Yuan and Connor R. Mooney for their helpful comments.
This work was partially supported by NSFC 12171389 and NSFC 11801015.

\renewcommand\refname{References}

\end{document}